\documentclass[12pt]{amsart}
\usepackage{amsmath}
\usepackage{amsfonts}
\usepackage{latexsym}
\usepackage{graphicx}
\usepackage{graphics}
\usepackage{amssymb}
\usepackage{color}

\theoremstyle{plain}
\newtheorem{theorem}{Theorem}[section]
\newtheorem{lemma}[theorem]{Lemma}

\newtheorem{observation}[theorem]{Observation}
\newtheorem{construction}[theorem]{Construction}
\newtheorem{corollary}[theorem]{Corollary}
\theoremstyle{definition}
\newtheorem{definition}[theorem]{Definition}
\newtheorem{remark}[theorem]{Remark}
\newtheorem{example}[theorem]{Example}

\begin{document}

\title{Automorphisms of Locally Compact Groups, Symbolic Dynamics and the Scale Function}
\author{Bruce~P.~Kitchens}
\address{Department of Mathematical Sciences, Indiana University - Purdue University Indianapolis, 402 N. Blackford Street,
Indianapolis, IN 46202}
\email{bkitchen@iu.edu}


\begin{abstract} It is shown how to model any automorphism of a totally disconnected, locally compact group by a symbolic dynamical system.  The model is an inverse limit of a product of a
full-shift, on a finite number of symbols, with one of two types of systems.  One is a countable discrete space with a permutation having every point periodic and the other is an essentially wandering, countable state Markov shift.  Some of the ideas used are from dynamics and some from the study of totally disconnected, locally compact groups.  The later ideas concern the scale function and tidy subgroups.  There is a discussion of the connections between those ideas and the dynamical ideas.  It is seen that only the essentially wandering, countable state Markov shift affects the scale function.  Finally, it's shown that transitivity or ergodicity with respect to Haar measure implies that the system has no countable discrete or essentially wandering component.
\end{abstract}

\maketitle

\section{Background and Motivation}  In \cite{K2} it was shown that any expansive automorphism of a second countable, totally disconnected, locally compact topological group is topologically conjugate to a product of a full-shift, on a finite number of symbols, with one of two types of systems.  One is a countable discrete space with a permutation having every point periodic and the other is an essentially wandering, countable state Markov shift.  In the present paper the results, with suitable modifications, are extended to arbitrary automorphisms of such a group.  This results in a symbolic dynamical structure theorem for such automorphisms. Then the connections between the symbolic dynamics ideas and the scale function and tidy subgroups, used in investigating the structure of totally disconnected, locally, compact groups, are discussed.  The scale function and tidy subgroups were introduced and investigated by G.~A.~Willis~\cite{W1}~\cite{W2}.  It is seen that the value of the scale function is determined by the essentially wandering component.  In the last section it is seen what this structure theorem implies when $T$ is transitive or ergodic with respect to Haar measure.

\section{Basic Structure} We are interested in the dynamics of automorphisms of topological groups that are locally compact and totally disconnected.  We abbreviate locally compact, totally disconnected by l.c.t.d.  For the rest of the paper we assume the groups are second countable, Hausdorff and the automorphisms are continuous.  First we state a basic result about l.c.t.d. groups.  It is a slight generalization of van Dantztig's Theorem  and the proof can easily be adapted from the one in \cite{HR} Section 7.

\begin{theorem}\label{CvanD}  Let $\mathcal{G}$ be a totally disconnected, locally compact group, $\mathcal{C}$ a compact subgroup of $\mathcal{G}$ and $U$ a neighborhood of $\mathcal{C}$, then there is a compact, open subgroup $\mathcal{H}$ with $\mathcal{C}\subseteq \mathcal{H} \subseteq U$. \end{theorem}

Any compact, open subgroup, $\mathcal{H}$, defines a coding to a symbolic system.  We will consider $\mathcal{G}/\mathcal{H}$ as the coset space of $\mathcal{G}$ and also as the countable discrete alphabet for a symbolic system.  Let $(\mathcal{G}/\mathcal{H})^\mathbb{Z}$ have the product topology.  Note that when $\mathcal{G}/\mathcal{H}$ is finite $(\mathcal{G}/\mathcal{H})^\mathbb{Z}$ is compact but when $\mathcal{G}/\mathcal{H}$ is infinite $(\mathcal{G}/\mathcal{H})^\mathbb{Z}$ is not even locally compact. There is a coding map, $\pi$, that takes $\mathcal{G}$ into $(\mathcal{G}/\mathcal{H})^\mathbb{Z}$ by assigning to every point in $\mathcal{G}$ its coset itinerary under $T$.  Let $X_{\mathcal{G}/\mathcal{H}} \subseteq (\mathcal{G}/\mathcal{H})^\mathbb{Z}$ denote $\pi(\mathcal{G})$.  Assume the standard symbolic dynamics definitions and notations for follower and predecessor sets for symbols and words occurring in $X_{\mathcal{G}/\mathcal{H}}$ (see for example \cite{K3}). The follower and predecessor sets of a coset in $\mathcal{G}$ are
\[
f(x\mathcal{H}) = \bigcup_{\{ y\/\mathcal{H}: T(x\mathcal{H}) \cap y\mathcal{H} \neq \phi\}}y\mathcal{H} \quad \text{and} \quad
p(x\mathcal{H}) = \bigcup_{\{ y\/\mathcal{H}: T^{-1}(x\mathcal{H}) \cap y\mathcal{H} \neq \phi\}}y\mathcal{H}.
\]
Since each coset is compact and open the cardinality of $f(\mathcal{H})$ is finite. Because all symbols correspond to cosets of $\mathcal{H}$ and all images under $T$ in $\mathcal{G}$ are cosets of $T(\mathcal{H})$ the cardinality of the follower set of any symbol is the same of the cardinality of $f(\mathcal{H})$.  The same applies to the predecessor and follower sets of the words of equal length from $X_{\mathcal{G}/\mathcal{H}}$.  All the cardinalities of the follower sets are bounded by the cardinality of $f(\mathcal{H})$ and all the cardinalities of the predecessor sets are bounded by the cardinality of $p(\mathcal{H})$.  This is the same as saying $X_{\mathcal{G}/\mathcal{H}}$ is locally compact.  Moreover, since the cardinality of $f(\mathcal{H}^{n+1})$, where $\mathcal{H}^{n+1}$ is the word of $n+1$ $\mathcal{H}$'s, is less than or equal to the cardinality of $f(\mathcal{H}^n)$ there is an $N$ where the cardinality stops decreasing.  Taking the standard higher block presentation produces a one-step, countable state Markov shift.  This can be formulated as follows.

\begin{observation}\label{countable}  Let $\mathcal{G}$ be an l.c.t.d. group and $T$ an automorphism.  If $\mathcal{H}$ is a compact, open subgroup of $\mathcal{G}$ then there is an $N \in \mathbb{N}$ so that the coding map using the compact, open subgroup
\[
\bigcap_{i = -N}^N T^{-i} (\mathcal{H})
\]
produces a factor which is a one-step, locally compact, countable state Markov shift.
\end{observation}

\begin{definition}\label{memoryless}  Let $\mathcal{G}$ be an l.c.t.d. group and $T$ an automorphism.  If $\mathcal{H}$ is a compact, open subgroup of $\mathcal{G}$ whose coding map produces a factor which is a one-step, locally compact, countable state Markov shift we say the partition produced by $\mathcal{H}$ is \textit{memoryless}.
\end{definition}

\begin{definition}\label{Markov} Let $\mathcal{G}$ be an l.c.t.d. group and $T$ an automorphism.  Let $\mathcal{H}$ be a compact, open subgroup of $\mathcal{G}$ which produces a partition that is memoryless. Denote by $(\Sigma_{\mathcal{G}/\mathcal{H}}, \sigma )$ the Markov shift produced by the coding map. \end{definition}

\begin{definition}\label{infinity} Let $\mathcal{G}$ be an l.c.t.d. group and $T$ an automorphism.  Let $\mathcal{H}$ be a compact, open subgroup of $\mathcal{G}$ which produces a partition that is memoryless.
Define the compact subgroup
\[
\mathcal{H}^\infty = \bigcap_{j \in \mathbb{Z}} T^j (\mathcal{H}).
\]
$\mathcal{H}^\infty$ consists of the points whose itinerary is all $\mathcal{H}$.  All points in each coset of $\mathcal{H}^\infty$are collapsed to a single point by the coding map.
\end{definition}

\begin{definition}\label{Mpartition} Let $\mathcal{G}$ be an l.c.t.d. group and $T$ an automorphism.  Let $\mathcal{H}$ be a compact, open subgroup of $\mathcal{G}$ which produces a partition that is memoryless. If $\mathcal{H}^\infty = \{e\}$ then the coset partition \textit{separates points} and the partition is a \textit{Markov partition}. \end{definition}

\begin{definition} Let $T$ be an automorphism of the l.c.t.d. group $\mathcal{G}$ and $\mathcal{H}$ be a compact, open subgroup whose coset partition is memoryless.
Define the subgroups
\[
\mathcal{H}_{loc}^s = \bigcap_{j \leq 0} T^j (\mathcal{H}) \qquad \text{and} \qquad \mathcal{H}_{loc}^u = \bigcap_{j \geq 0} T^j (\mathcal{H}).
\]
In dynamical terms these are the \textit{local stable set of the identity} and the \textit{local unstable set of the identity}.
\end{definition}

Note that the cosets $x \mathcal{H}_{loc}^s$ and $x \mathcal{H}_{loc}^u$, for all points $x \in \mathcal{G}$, are dynamically defined.  Each coset of $\mathcal{H}_{loc}^s$ is an equivalence class, in the corresponding coset of
$\mathcal{H}$, defined by the equivalence relation $x \sim_s y$ if and only if $T^j(x)$ and $T^j(y)$ are in the same element of the partition for all $j \geq 0$.  Each coset of $\mathcal{H}_{loc}^u$ is an equivalence class, in the corresponding coset of $\mathcal{H}$, defined by the equivalence relation $x \sim_u y$ if and only if $T^j(x)$ and $T^j(y)$ are in the same element of the partition for all $j \leq 0$.  Also, observe that
\[
\mathcal{H}_{loc}^s \cap \mathcal{H}_{loc}^u = \mathcal{H}^\infty.
\]

The following lemmas will be used in the proofs of the main theorems.  The first is a well known result and is stated without proof.  It and related ideas are discussed in Section 8.

\begin{lemma}\label{subgroup} Let $T$ be an automorphism of the l.c.t.d. group $\mathcal{G}$ and $\mathcal{H}$ be a compact, open subgroup.  Then the following conditions are equivalent.
\begin{enumerate}
\item  The coset partition by $\mathcal{H}$ is memoryless.
\item $\mathcal{H}  = \mathcal{H}^u_{loc} \mathcal{H}^s_{loc}  =\mathcal{H}^s_{loc} \mathcal{H}^u_{loc}$.
\item  $T(x \mathcal{H}^s_{loc}) \subseteq T(x) \mathcal{H}^s_{loc}$ and $T(x) \mathcal{H}^u_{loc} \subseteq T(x \mathcal{H}^u_{loc})$, for all $x \in \mathcal{G}$.
\end{enumerate}
\end{lemma}

\begin{lemma}\label{ptimesf} Let $T$ be an automorphism of the l.c.t.d. group $\mathcal{G}$ and $\mathcal{H}$ be a compact, open subgroup whose coset partition is memoryless then
\[
T(p(\mathcal{H})) = f(\mathcal{H}) \qquad T^{-1}(f(\mathcal{H})) = p(\mathcal{H})
\]
\[
\begin{array}{lllll}
f(\mathcal{H})& = T(\mathcal{H}) \mathcal{H} & = T(\mathcal{H}^u_{loc}) \mathcal{H}^s_{loc} & =\mathcal{H}^s_{loc} T(\mathcal{H}^u_{loc}) & = \mathcal{H} T(\mathcal{H}) \\
p(\mathcal{H})& = T^{-1}(\mathcal{H}) \mathcal{H} & = T^{-1}(\mathcal{H}^s_{loc}) \mathcal{H}^u_{loc} & =\mathcal{H}^u_{loc}T^{-1}(\mathcal{H}^u_{loc}) & = \mathcal{H} T^{-1}(\mathcal{H}).
\end{array}
\]
In particular, $f(\mathcal{H})$ and $p(\mathcal{H})$ are subgroups of $\mathcal{G}$.
\end{lemma}

\begin{proof} By Lemma~\ref{subgroup}, $f(\mathcal{H}) = T(\mathcal{H}^u_{loc}) \mathcal{H}^s_{loc}$ and $p(\mathcal{H}) = T^{-1}(\mathcal{H}^s_{loc}) \mathcal{H}^u_{loc}$.  Let $x \in p(\mathcal{H})$.
Then $xT^{-1}(\mathcal{H}^s_{loc}) \cap \mathcal{H}^u_{loc} \neq \phi$ by Lemma~\ref{subgroup}.  For $y$ in the intersection, $yT^{-1}(\mathcal{H}^s_{loc}) = xT^{-1}(\mathcal{H}^s_{loc})$, and
$x \in \mathcal{H}^u_{loc}T^{-1}(\mathcal{H}^s_{loc})$.  This means $T(x) \in f(\mathcal{H})$ so $T(p(\mathcal{H})) \subseteq f(\mathcal{H})$.  Likewise, $T^{-1}(f( \mathcal{H})) \subseteq p(\mathcal{H})$, which means
$T(p(\mathcal{H})) = f(\mathcal{H})$ and $T^{-1}(f( \mathcal{H})) = p(\mathcal{H})$.

The rest of the statements follow from this. \end{proof}

\section{Essentially Wandering Systems}

\begin{definition}\label{ewand} Let $X$ be an infinite topological space and $T$ a homeomorphism from $X$ to itself.  The system $(X,T)$ is \textit{essentially wandering} if:
\begin{enumerate}
\item the number of points in $X$ with a compact $T$-orbit closure is countable and each point is periodic;
\item it contains points with a noncompact orbit closure.
\end{enumerate}
\end{definition}

\begin{example}\label{1101} Let $\mathcal{G} = \mathbb{Z}^2$ and $T$ be defined by the matrix
\[
\left[ \begin{array}{rr} 1 & 1 \\ 0 & 1 \end{array} \right].
\]
Every point $(0,n)$ for $n \in \mathbb{Z}$ is a fixed point for $T$ but every other point has a noncompact orbit closure.
\end{example}

\begin{example}\label{3adic} Consider the 3-adic numbers, $\mathbb{Q}_3$, expressed as the subset of $\{0,1,2\}^\mathbb{Z}$ consisting of the sequences that are eventually all zeros to the left. A group operation is defined by coordinate by coordinate addition in base 3 with a carry to the right.  The automorphism, $\sigma^{-1}$, is multiplication by 3 which is the inverse of the usual shift on the sequences.  If we choose for $\mathcal{H}$ the subgroup of sequences that are zero for all coordinates to the left of the time zero entry, we produce the essentially wandering, countable state Markov shift defined by the transition graph below where $\bar{0}$ means zeros from that coordinate to the left.


\begin{figure}[h]
\scalebox{0.5} {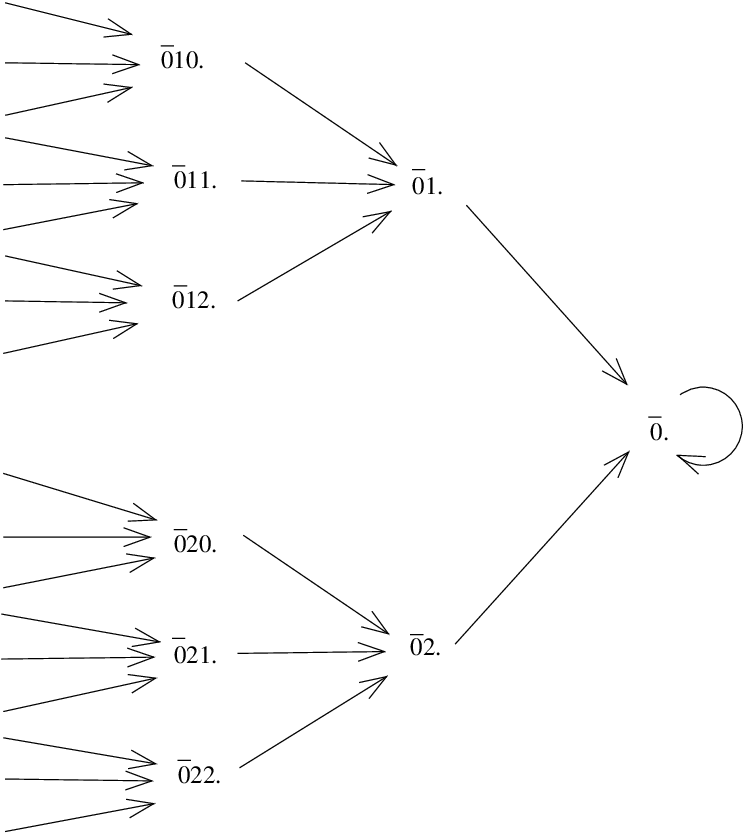}
\end{figure}

\end{example}

\begin{example}\label{3adicsq} Consider $\mathbb{Q}_3 \oplus \mathbb{Q}_3$ the direct sum of two copies of the 3-adic numbers.  Express the entries in each copy of $\mathbb{Q}_3$ as in Example~\ref{3adic}.  The automorphism is multiplication by 1/3 in the first copy of $\mathbb{Q}_3$ and multiplication by 3 in the second copy.  Take $\mathcal{H}$ to be the subgroup $\{ ( \bar{0}. , \bar{0}. ) \}$ in the previous notation.  This gives the essentially wandering, countable state Markov shift defined by the transition graph below.


\begin{figure}[h]
\scalebox{0.5} {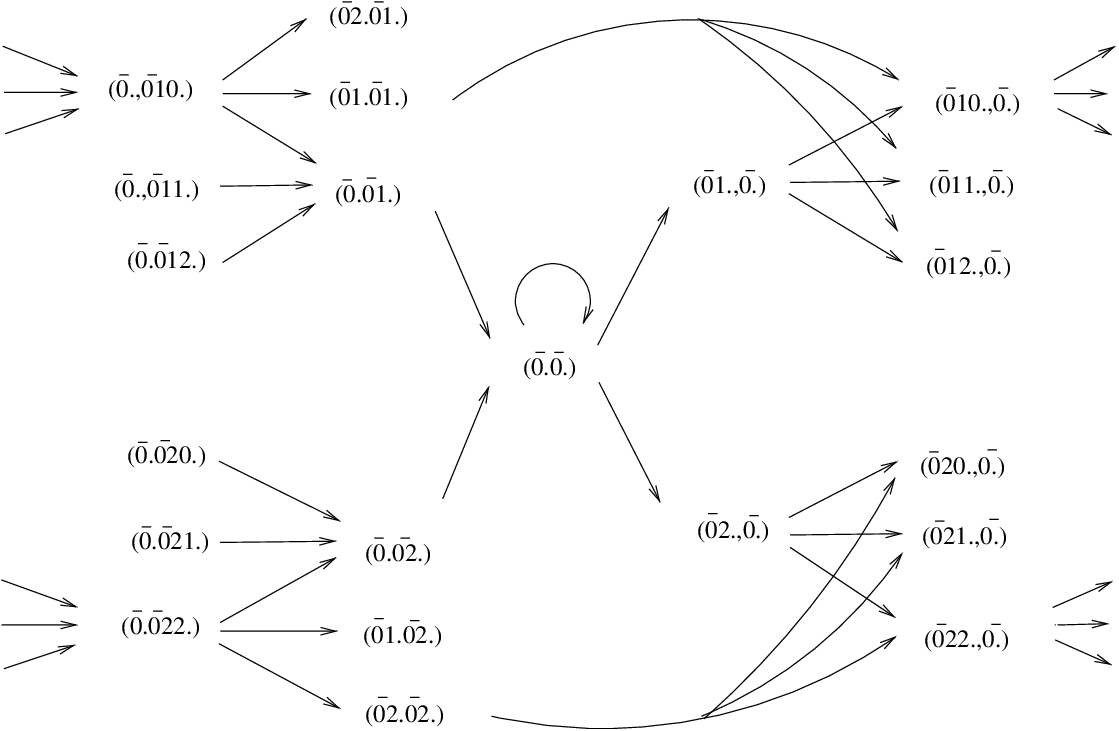}
\end{figure}

\end{example}

\begin{observation}\label{tree}  Let $X$ be an essentially wandering countable state Markov shift.  If $x\mathcal{H}$ and $y\mathcal{H}$ contain distinct periodic orbits there is no point in $x\mathcal{H}$ whose orbit
intersects $y\mathcal{H}$.
\end{observation}

\begin{proof} Since $X$ is a Markov shift, if this were not true there would be points with a compact orbit closure that are not periodic in $X$.
\end{proof}

\section{Dense Compact Orbit Closures}

The opposite end of the spectrum from essentially wandering systems are the ones where the points with a compact $T$-orbit closure are dense in $\mathcal{G}$.  Some examples follow.

\begin{example}\label{gpshift} Let $G$ be a finite  group with the discrete topology.  The product space, $G^\mathbb{Z}$, with the product topology and group multiplication defined coordinate by coordinate is a compact, totally disconnected group.  The shift transformation, $\sigma$, defined by $\sigma(x)_i=x_{i+1}$ is a group automorphism.
\end{example}

\begin{example}\label{01-10} Let $\mathcal{G} = \mathbb{Z}^2$ and $T$ be defined by the matrix
\[
\left[ \begin{array}{rr} 0 & 1 \\ -1 & 0 \end{array} \right]
\]
which is rotation by $\pi/2$.  Every point is periodic with $(0,0)$ a fixed point and all others of period four.

\end{example}

\begin{example}\label{dlim} Let $\mathcal{G}$ be the direct sum of the groups $\mathbb{Z}/ 3^n \mathbb{Z}$, for $n \in \mathbb{N}$,
\[
\mathcal{G} = \bigoplus_{n \in \mathbb{N}}  \mathbb{Z}/ 3^n \mathbb{Z}.
\]
The space is the subgroup of $\prod_{n \in \mathbb{N}} \mathbb{Z}/ 3^n \mathbb{Z}$ where all sequences $x \in \mathcal{G}$ have $x_n =0$ for all but finitely many $n \in \mathbb{N}$.  The group is countably infinite and we put the discrete topology on it.  Let $T$ be the automorphism defined using multiplication by $2$ on each $\mathbb{Z}/ 3^n \mathbb{Z}$.  Every point is periodic but there are points of arbitrarily high period.
\end{example}

\begin{example}\label{cptocl} Let $S_3$ be the permutation group of a set containing three elements.   Form the full-shift $((S_3)^\mathbb{Z}, \sigma)$ and let $(\mathcal{G},T)$ be the system from Example~\ref{dlim}.  Form the direct sum system $( (S_3)^\mathbb{Z} \oplus \mathcal{G}, \sigma \times T)$.  Now, $(S_3)^\mathbb{Z} \oplus \mathcal{G}$ is a nonabelian l.c.t.d. group, $\sigma \times T$ is an expansive automorphism and the periodic points are dense in the group.
\end{example}

\begin{example}\label{compact}  Define $\Sigma_A \subseteq ( \mathbb{Z} / 4 \mathbb{Z} \oplus \mathbb{Z} / 2 \mathbb{Z})^\mathbb{Z}$ by the transition graph below.

\begin{center}
\includegraphics[width=1in]{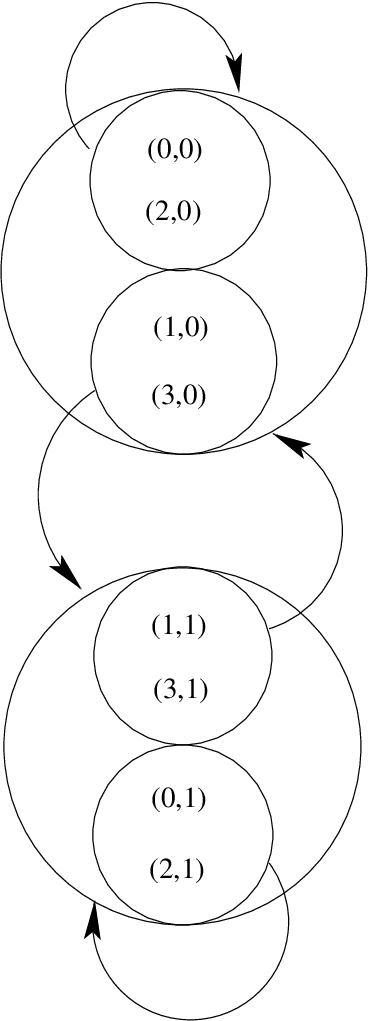}
\end{center}


This means
\[
\begin{array}{c}
f((0,0)) = f((2,0)) = f((1,1)) = f((3,1)) = \{ (0,0),(2,0),(1,0),(3,0) \} \\
f((1,0)) = f((3,0)) = f((0,1)) = f((2,1)) = \{ (0,1),(2,1),(1,1),(3,1) \}.
\end{array}
\]
In this case the finite state Markov shift is a compact group where the group operation is coordinate by coordinate addition (with no carry).  The subgroup $p(\mathcal{H}) \cap f(\mathcal{H}) = \mathcal{K}$ is $\{ (0,0), (2,0) \}$.  The Markov shift $(\Sigma_A, \sigma)$ is topologically conjugate to the full-shift on four symbols but it is not algebraically conjugate to either
$((\mathbb{Z}/ 2\mathbb{Z} \oplus \mathbb{Z}/2\mathbb{Z})^\mathbb{Z}, \sigma)$ or $((\mathbb{Z}/4\mathbb{Z})^\mathbb{Z}, \sigma)$.
\end{example}

\section{Two Constructions}

Next we describe two constructions that will be used in the proof of Theorem~\ref{theorem1}.  They were originally used in \cite{K1}.

\begin{definition}\label{subgp K} Let $T$ be an automorphism of the l.c.t.d. group $\mathcal{G}$ and $\mathcal{H}$ be a compact, open subgroup whose coset partition is memoryless.  Denote by $\mathcal{K}$ the nonempty, compact, open subset of $\mathcal{G}$ defined by $\mathcal{K} = p(\mathcal{H}) \cap f(\mathcal{H})$.  By Lemma~\ref{ptimesf} $f(\mathcal{H})$ and $p(\mathcal{H})$ are subgroups of $\mathcal{G}$ and so $\mathcal{K}$ is also a subgroup of $\mathcal{G}$.
\end{definition}

The cosets of $\mathcal{H}$ partition $\mathcal{K}$ and all $\mathcal{H}$ cosets in $\mathcal{K}$ have the same follower and the same predecessor sets.  This is also true for each coset of $\mathcal{K}$.  We say the cardinality of $\mathcal{K}$ is the number of cosets of $\mathcal{H}$ it contains.

\vskip 0.2in

\begin{construction}\label{con1} The case when the cardinality of $\mathcal{K}$ is greater than one. \end{construction}

Consider the subgroups of $\mathcal{G}$, $p(\mathcal{H})$, $f(\mathcal{H})$, $\mathcal{K}$ and $\mathcal{H}$.  Observe that the coset partition defined by $\mathcal{K}$ is memoryless, $\mathcal{K}$ is partitioned into cosets of $\mathcal{H}$ and $\mathcal{K}^\infty = \cap _{j \in \mathbb{Z}} T^j (\mathcal{K})$ is partitioned into cosets of $\mathcal{H} \cap \mathcal{K}^\infty$.  $T$ restricted to the subgroup $\mathcal{K}^\infty$, using the coset partition defined by $\mathcal{H} \cap \mathcal{K}^\infty$, codes to a full shift on $|\mathcal{K}$/$\mathcal{H}|$ elements.

Build two new systems, $\Sigma_{\mathcal{G}/\mathcal{K}}$ and $(\mathcal{K}/\mathcal{H})^\mathbb{Z}$.  The space $\Sigma_{\mathcal{G}/\mathcal{K}}$ (when not countable) has all the properties of
$\Sigma_{\mathcal{G}/\mathcal{H}}$.  The point is that the systems $\Sigma_{\mathcal{G}/\mathcal{H}}$ and $\Sigma_{\mathcal{G}/\mathcal{H}} \times (\mathcal{K}/\mathcal{H})^\mathbb{Z}$ are topologically conjugate.  The conjugacy is defined on the symbol level.  It is clear because there is the algebraic map $\mathcal{G} / \mathcal{H} \rightarrow \mathcal{G} / \mathcal{K}$ with ``kernel" $\mathcal{K} / \mathcal{H}$.  If $\mathcal{H}$ and $\mathcal{K}$ are normal subgroups of $\mathcal{G}$ this is just saying that $\mathcal{G} / \mathcal{H}$ is an extension of $\mathcal{K} / \mathcal{H}$ by $\mathcal{G} / \mathcal{K}$ so every element of $\mathcal{G} / \mathcal{H}$ can be written as a pair of elements with the first from $\mathcal{G} / \mathcal{K}$ and the second from $\mathcal{K} / \mathcal{H}$. We have reduced the Markov shift $\Sigma_{\mathcal{G}/\mathcal{H}}$ to a product of a new Markov shift cross a full-shift on a finite number of symbols.  In the new Markov shift the cardinalities of the predecessor and follower sets are strictly smaller than in the old one.  The new cardinalities of the predecessor and follower sets is the old cardinality (which was finite) divided by the cardinality of $\mathcal{K}$.  Likewise, the cardinality of the alphabet of $\Sigma_{\mathcal{G}/\mathcal{K}}$ is the cardinality of the alphabet of $\Sigma_{\mathcal{G}/\mathcal{H}}$ divided by the cardinality of $\mathcal{K}$.  If the alphabet of $\Sigma_{\mathcal{G}/\mathcal{H}}$ was infinite the new alphabet is infinite. \hfill $\Box$

\vskip 0.2in

\begin{construction}\label{con2} The case when the cardinality of $\mathcal{K}$ is one. \end{construction}

Consider the subgroups of $\mathcal{G}$, $p(\mathcal{H})$ and $f(\mathcal{H})$.  Both subgroups contain $\mathcal{H}$.  The important point is that when the cardinality of $\mathcal{K}$ is one every element of
$f(\mathcal{H})$ has a distinct follower set and every element of $p(\mathcal{H})$ has a distinct predecessor set.  To see this assume two elements of $f(\mathcal{H})$ have the same follower set.  Using the fact that $f(\mathcal{H})$ is a subgroup we can assume that one of them is $\mathcal{H}$ and the other is $x\mathcal{H}$.  But then we have both $\mathcal{H}$ and $x\mathcal{H}$ in $\mathcal{K}$.  The same reasoning applies to the predecessor sets.  In turn, each element of any follower set, $f(x\mathcal{H})$, has a unique follower set and each element of any predecessor set, $p(x\mathcal{H})$, has a unique predecessor set.

Both of the subgroups $p(\mathcal{H})$ and $f(\mathcal{H})$ have coset partitions of $\mathcal{G}$ that are memoryless and define one-step Markov shifts $\Sigma_{\mathcal{G}/f(\mathcal{H})}$ and $\Sigma_{\mathcal{G}/p(\mathcal{H})}$.  Each is a factor of $\Sigma_{\mathcal{G}/\mathcal{H}}$ by a map defined on symbols.  Moreover, each factor map is invertible.  The map onto $\Sigma_{\mathcal{G}/f(\mathcal{H})}$ is invertible by a two-block map that looks at the present symbol and one symbol into the future.  This is well-defined because each element of a fixed follower set has a distinct follower set.  Likewise, the map onto $\Sigma_{\mathcal{G}/p(\mathcal{H})}$ is invertible by a two-block map that looks at the present symbol and one symbol into the past. We have shown that $(\Sigma_{\mathcal{G}/\mathcal{H}}, \sigma)$,
$(\Sigma_{\mathcal{G}/p(\mathcal{H})}, \sigma)$ and $(\Sigma_{\mathcal{G}/f(\mathcal{H})}, \sigma)$ are all topologically conjugate.

 The number of distinct follower sets in $\Sigma_{\mathcal{G}/f(\mathcal{H})}$ is the number of distinct follower sets in $\Sigma_{\mathcal{G}/\mathcal{H}}$ divided by the cardinality of $f(\mathcal{H})$ and the number of distinct predecessor sets in $\Sigma_{\mathcal{G}/p(\mathcal{H})}$ is the number of distinct predecessor sets in $\Sigma_{\mathcal{G}/\mathcal{H}}$ divided by the cardinality of $p(\mathcal{H})$.  The cardinality of the alphabet of $\Sigma_{\mathcal{G}/f(\mathcal{H})}$ is the cardinality of the alphabet of $\Sigma_{\mathcal{G}/\mathcal{H}}$ divided by the cardinality of $f(\mathcal{H})$ and the cardinality of the alphabet of $\Sigma_{\mathcal{G}/p(\mathcal{H})}$ is the cardinality of the alphabet of $\Sigma_{\mathcal{G}/\mathcal{H}}$ divided by the cardinality of $p(\mathcal{H})$.  If the alphabet of $\Sigma_{\mathcal{G}/\mathcal{H}}$ was infinite the new alphabets are infinite.  A crucial point is that whether we use follower or predecessor sets to make the reduction the cardinalities of the new predecessor and follower sets, in both cases, are unchanged.  The cardinalities neither increase nor decrease.  \hfill $\Box$

The two constructions can be best understood by applying them to Example~\ref{compact}.  First apply Construction~\ref{con1} and then Construction~\ref{con2} to obtain $(\Sigma_2, \sigma) \times (\Sigma_2, \sigma)$ which is topologically conjugate to $(\Sigma_4, \sigma)$.

\section{Structure Theorems}

\begin{definition}\label{ccomp} Let $T$ be an automorphism of an l.c.t.d. group $\mathcal{G}$.  Define $\mathcal{G}_c$ to be the closure of the points whose $T$-orbit closure is compact in $\mathcal{G}$.  Observe that $\mathcal{G}_c$ is a closed, $T$-invariant subgroup of $\mathcal{G}$.  Denote by $T_c$ the restriction of $T$ to $\mathcal{G}_c$.
\end{definition}

\begin{definition}\label{ccomp} Let $T$ be an automorphism of an l.c.t.d. group $\mathcal{G}$ and $\mathcal{H}$ a subgroup whose coset partition is memoryless.  Define $\mathcal{G}_t$ to be the closure of the points $x \in \mathcal{G}$, where there is an $N \in \mathbb{N}$ and for all $|n| \geq \mathbb{N}$, $T^n(x) \in \mathcal{H}$.  Observe that $\mathcal{G}_t$ is a closed, $T$-invariant subgroup of $\mathcal{G}$.  Denote by $T_t$ the restriction of $T$ to $\mathcal{G}_t$.
\end{definition}

We have $\mathcal{G}_t \subseteq \mathcal{G}_c \subseteq \mathcal{G}$ and each is an l.c.t.d. group with automorphisms $T_t$, $T_c$ and $T$.

\begin{theorem}\label{theorem1}  Let $T$ be an automorphism of the l.c.t.d. group $\mathcal{G}$ and $\mathcal{H}$ be a compact, open subgroup whose coset partition is memoryless.
\begin{enumerate}
\item If the points with compact $T$-orbit closure are dense in $\mathcal{G}$ then $(\Sigma_{\mathcal{G}/\mathcal{H}}, \sigma)$ is topologically conjugate to the product of a full-shift on a finite alphabet and a
permutation of a countable discrete set where every orbit is finite.
\item If the points with compact $T$-orbit closure are not dense in $\mathcal{G}$ then $(\Sigma_{\mathcal{G}/\mathcal{H}}, \sigma)$ is topologically conjugate to the product of a full-shift on a finite alphabet
and an essentially wandering countable state Markov shift.
\end{enumerate}
\end{theorem}

\begin{proof}
\textit{Statement 1.}  $\mathcal{G}_c = \mathcal{G}$.  Use the coding map from the coset partition by $\mathcal{H}$ to produce the countable state Markov shift $(\Sigma_{\mathcal{G}/\mathcal{H}}, \sigma)$.

By Lemma~\ref{ptimesf} we know that $f(\mathcal{H})$ and $p(\mathcal{H})$ are compact, open subgroups of $\mathcal{G}$.  Consider the subgroup $\mathcal{K}$ defined in Definition~\ref{subgp K}.

If the cardinality of $\mathcal{K}$ is greater than one apply Construction~\ref{con1}.  This produces $\Sigma_{\mathcal{G}/\mathcal{K}} \times (\mathcal{K}/\mathcal{H})^\mathbb{Z}$ topologically conjugate to $\Sigma_{\mathcal{G}/\mathcal{H}}$.  We think of factoring out a full-shift on a finite number of symbols.  As observed in Construction~\ref{con1} we are changing the subgroup from $\mathcal{H}$ to $\mathcal{K}$ to obtain the countable state Markov shift $\Sigma_{\mathcal{G}/\mathcal{K}}$ and using the coset partition of $\mathcal{H} \cap \mathcal{K^\infty}$ in $\mathcal{K}^\infty$ to obtain $(\mathcal{K}/\mathcal{H})^\mathbb{Z}$.
We can ignore the full-shift and concentrate on the partition of $\mathcal{G}$ defined by $\mathcal{K}$.  The important point is that in $\Sigma_{\mathcal{G}/\mathcal{K}}$ the cardinality of the predecessor and follower sets are strictly smaller than in $\Sigma_{\mathcal{G}/\mathcal{H}}$.  The new cardinalities of the predecessor and follower sets is the old cardinalities (which was finite) divided by the cardinality of $\mathcal{K}$.

If the cardinality of $\mathcal{K}$ is one, use $p(\mathcal{H})$ or $f(\mathcal{H})$ to define a new partition of $\mathcal{G}$.  This means applying Construction~\ref{con2}, identifying the $\mathcal{H}$ cosets
in $p(\mathcal{H})$ or $f(\mathcal{H})$ to obtain the topologically conjugate countable state Markov shift $\Sigma_{\mathcal{G}/p(\mathcal{H})}$ or $\Sigma_{\mathcal{G}/f(\mathcal{H})}$.  The cardinalities of
follower and predecessor sets are the same in the new Markov shift as in $\Sigma_{\mathcal{G}/\mathcal{H}}$.

Continue to apply Constructions~\ref{con1} or \ref{con2} as appropriate.  Notice that since the cardinalities of $p(\mathcal{H})$ and $f(\mathcal{H})$ were originally finite, Construction~\ref{con1} can be used only a finite number of times.

The cardinalities of $p(\mathcal{H})$ and $f(\mathcal{H})$ will be reduced to one.  To see this, suppose the cardinality of $f(\mathcal{H})$ is greater than one and the cardinality of $\mathcal{K}$ is one.  Then
since the points with compact $T$-orbit closure are dense in $\mathcal{G}$ there is either a periodic point not equal to the identity element or a preperiodic point in $\mathcal{H}$.  Apply Construction~\ref{con2}
repeatedly until the cardinality of $\mathcal{K}$ is greater than one and then apply Construction~\ref{con1} further reducing the cardinality of $f(\mathcal{H})$.  The cardinality of $\mathcal{K}$ will become greater
than one after a finite number of applications of Construction~\ref{con2} because of the periodic point or preperiodic point.  More and more of the orbit is getting absorbed into $\mathcal{H}$  as Construction~\ref{con2}
is applied. Finally, there will be transitions that go from $\mathcal{H}$ to a distinct $g\mathcal{H}$ and then from this $g\mathcal{H}$ back to $\mathcal{H}$ making the cardinality of $\mathcal{K}$ greater than one.

If the cardinality of $p(\mathcal{H})$ is greater than one we make the analogous constructions.

Now we have a countable state Markov shift, whose predecessor and follower sets have cardinality one, cross a full-shift on a finite number of symbols.  A countable state Markov shift where every predecessor and follower set
has cardinality one and where every obit closure is compact consists of a countable number of periodic orbits.

Observe that the full-shift is topologically conjugate to the image under the coding map of $(\mathcal{G}_t, T_t)$.  If we return to $\mathcal{G}$ and form $\mathcal{G}/\mathcal{G}_t$ we have a coset space with an induced transformation whose image under the coding map is topologically conjugate to the Markov shift constructed.  The Markov shift is countable, discrete set and every orbit under the induced permutation is finite.  We have that $(\Sigma_{\mathcal{G}/\mathcal{H}}, \sigma)$ is topologically conjugate to the product of a full-shift on a finite alphabet and a permutation of a countable discrete set where every orbit is finite.
\end{proof}

\begin{corollary}\label{cmodt} Let $T$ be an automorphism of the l.c.t.d. group $\mathcal{G}$ and $\mathcal{H}$ be a compact, open subgroup whose coset partition is memoryless.  Denote by
$\Sigma_{\mathcal{G}_t / \mathcal{H}}$ the image of the subgroup $\mathcal{G}_t$ in $\Sigma_{\mathcal{G}/ \mathcal{H}}$ and $\Sigma_{\mathcal{G}_c / \mathcal{H}}$ the image of the subgroup $\mathcal{G}_c$ in $\Sigma_{\mathcal{G}/ \mathcal{H}}$.  $\Sigma_{\mathcal{G}_t/ \mathcal{H}}$ is compact and $(\Sigma_{\mathcal{G}_t / \mathcal{H}}, \sigma)$ is topologically conjugate to a full-shift on a finite number of symbols.  The
system $(\Sigma_{\mathcal{G}_c / \mathcal{H}}, \sigma)$ is topologically conjugate to
\[
(\Sigma_{\mathcal{G}_t / \mathcal{H}}, \sigma) \times (F, \tau)
\]
where $F$ is a countable set and $\tau$ is a permutation of $F$ where every orbit is finite.
\end{corollary}

\begin{corollary}\label{cptcl} Let $T$ be an automorphism of the l.c.t.d. group $\mathcal{G}$ where the points with a compact $T$-orbit closure are dense in $\mathcal{G}$. If $\mathcal{H}$ is a compact, open subgroup whose coset partition is memoryless, its coding map produces
\[
(\Sigma_{\mathcal{G}_t / \mathcal{H}}, \sigma) \times (F, \tau)
\]
where where every point has a compact $T$-orbit closure.
\end{corollary}

\begin{observation}\label{tint}  Another interpretation of $\mathcal{G}_t$ can be given.  Let $F_t$ be the finite collection of all cosets of $\mathcal{H}$ that contain a periodic point whose
$T$-orbit passes through $\mathcal{H}$.  Define
\[
C_{F_t} = \bigcup_{g\mathcal{H} \in F_t} g\mathcal{H} \quad \text{then} \quad \mathcal{G}_t = \bigcap_{j \in \mathbb{Z}} T^j(C_{F_t}).
\]
\end{observation}

\begin{proof}
\textit{Statement 2.}  Consider $\mathcal{G}/\mathcal{G}_t$ with the quotient topology and the induced automorphism $T^\prime $.  Let $\rho$ denote the quotient map from $\mathcal{G}$ to $\mathcal{G}/\mathcal{G}_t$.
If $E \subseteq \mathcal{G}$ then $\rho^{-1}(\rho(E)) = E \mathcal{G}_t$.  If $U \subseteq \mathcal{G}$ is open then $\rho^{-1}(\rho(U)) = U \mathcal{G}_t$ is open in $\mathcal{G}$ and consequently $\rho(U)$ is open in the quotient topology on $\mathcal{G}/\mathcal{G}_t$.  This means $\rho$ is an open map and $\mathcal{G}/\mathcal{G}_t$ is a second countable, Hausdorff, locally compact and totally disconnected coset space.

Consider the compact, open set $C_{F_t}$ defined in Observation~\ref{tint}.  It contains $\mathcal{G}_t$ and by applying Theorem~\ref{CvanD} and Observation~\ref{countable} we can produce a compact, open subgroup
$\mathcal{U} \subseteq \mathcal{G}$, whose coset partition is memoryless and with $\mathcal{G}_t \subseteq \mathcal{U} \subseteq C_t$.  By construction
\[
\bigcap_{j \in \mathbb{Z}} T^j (\mathcal{U}) = \mathcal{G}_t
\]
so $( \Sigma_{\mathcal{G}/\mathcal{U}}, \sigma)$ and $(\mathcal{G}/\mathcal{G}_t , T^\prime)$ are topologically conjugate.

By Corollary~\ref{cmodt} the coset space $\mathcal{G}_c/ \mathcal{G}_t$ consists of a countable number of periodic points and is a $T^\prime$-invariant subset of the coset space $\mathcal{G}/ \mathcal{G}_t$.
A point $x \mathcal{G}_t \in \mathcal{G}/ \mathcal{G}_t$ which has a compact orbit closure has preimages in $\mathcal{G}$ with compact orbit closure since $\mathcal{G}_t$ is compact.  Therefore, it is in $\mathcal{G}_c/ \mathcal{G}_t$ and so
$( \mathcal{G}/ \mathcal{G}_t, T^\prime)$ is essentially wandering.  Together with Corollary~\ref{cmodt}, this proves (2).
\end{proof}

\begin{remark}\label{11101b}  Consider Example~\ref{1101}.  In the example $\mathcal{G}_t = \{ (0,0) \}$, $\mathcal{G}_c = \{ (0,n): n \in \mathbb{Z} \}$.  $\mathcal{G}/\mathcal{G}_c$ is isomorphic to $\mathbb{Z}$ where the induced map is the identity.  $(\mathcal{G},T)$ cannot be reduced to a product of $(\mathcal{G}_c, T_c)$ and another system.
\end{remark}

\begin{theorem}\label{theorem2}  Let $T$ be an automorphism of the l.c.t.d. group $\mathcal{G}$.
\begin{enumerate}
\item If the points with compact $T$-orbit closure are dense in $\mathcal{G}$ then $(\mathcal{G}, T)$ is topologically conjugate to an inverse limit of systems that are the product of a full-shift on a
finite alphabet and a permutation of a countable discrete set where every orbit is finite.
\item If the points with compact $T$-orbit closure are not dense in $\mathcal{G}$ then $(\mathcal{G}, T)$ is topologically conjugate to an inverse limit of systems that are the product of a full-shift on a
finite alphabet and an essentially wandering countable state Markov shift.
\end{enumerate}
\end{theorem}

\begin{proof} Choose a nested sequence of compact, open subgroups $\{\mathcal{H}_n \}_{n \in \mathbb{N}}$ such that
\begin{enumerate}
\item the coset partition of each $\mathcal{H}_n$ is memoryless, and
\item $\bigcap_{n \in \mathbb{N}} \mathcal{H}_n = \{e\}$.
\end{enumerate}
Apply Theorem~\ref{theorem1} at each $n$ to produce $(\Sigma_{\mathcal{G}/\mathcal{H}_n}, \sigma)$ which is conjugate to either the product of a full-shift on a finite alphabet and a permutation of a countable
discrete set where every orbit is finite or the product of a full-shift on a finite alphabet and an essentially wandering countable state Markov shift.

The bonding map at each level is a 1-block factor map defined by inclusion of the cosets.
\end{proof}

\section{The Expansive Case}

An automorphism $T$, of an l.c.t.d group $\mathcal{G}$, is \textit{expansive} if there is an open neighborhood, $\mathcal{U}$, of the identity so that for any two distinct points, $x,y \in \mathcal{G}$ there
is an $n \in \mathbb{Z}$ where $T^n(x) \notin T^n(y) \mathcal{U}$.  When an automorphism is expansive the previous results can be strengthened.

\begin{lemma}\label{tnormal}  If $T$ is an expansive an expansive automorphism of an l.c.t.d group $\mathcal{G}$ there is a compact, open subgroup $\mathcal{H}$ whose coset partition separates points and
$\mathcal{G}_t$ is normal in $\mathcal{G}$.
\end{lemma}

\begin{proof}  By Theorem~\ref{CvanD} there is a compact, open subgroup $\mathcal{H}$ so that
\[
\bigcap_{j \i \mathbb{Z}} T^j(\mathcal{H}) = \{e\}.
\]
Then $\mathcal{G}_t$ is the closure of the of the points where there is an $N \in \mathbb{N}$ such that $T^n(x) \in \mathcal{H}$ for all $|n| \geq N$.  Since the group operation is continuous this is a $T$-invariant, closed, normal subgroup of $\mathcal{G}$.
\end{proof}

\begin{theorem}\label{etheorem}  Let $T$ be an expansive automorphism of an l.c.t.d. group $\mathcal{G}$. The system falls into one of the two following categories.
\begin{enumerate}
\item $\mathcal{G}_c = \mathcal{G}$ in which case $(\mathcal{G}, T)$ is topologically conjugate to the product of a full-shift on a finite alphabet and an automorphism of a
countable discrete group where every orbit is finite.
\item $\mathcal{G}_c \neq \mathcal{G}$ in which case $(\mathcal{G}, T)$ is topologically conjugate to the product of a full-shift on a finite alphabet and an essentially wandering countable state Markov shift.
\end{enumerate}
\end{theorem}

\begin{proof}  By Lemma~\ref{tnormal} there is a compact, open, normal subgroup $\mathcal{H}$ which separates points, so that its coset partition is a Markov partition (Definition~\ref{Mpartition}).  Then
the coding map from $(\mathcal{G}, T)$ to $( \Sigma_{\mathcal{G}/\mathcal{H}} , \sigma)$ is a topological conjugacy.
\end{proof}

\begin{corollary}\label{cmodt2} Let $T$ be an expansive automorphism of an l.c.t.d. group $\mathcal{G}$. The subgroup $\mathcal{G}_t$ is compact and the system $(\mathcal{G}_t, T_t )$ is topologically conjugate to a full-shift on a finite number of symbols.  Moreover, $\mathcal{G}_t$ is open in $\mathcal{G}_c$, $\mathcal{G}_c / \mathcal{G}_t$ is a countable, discrete group and every orbit under the induced automorphism is periodic.
\end{corollary}

\begin{corollary}\label{cmodt2} Let $T$ be an expansive automorphism of an l.c.t.d. group $\mathcal{G}$. In the subgroup $\mathcal{G}_c$, every point has a compact $T$-orbit closure.
\end{corollary}

\begin{corollary} Let $T$ be an expansive automorphism of an l.c.t.d. group $\mathcal{G}$. If $T$ has a dense orbit in $\mathcal{G}$ then $\mathcal{G}$ is compact and $(\mathcal{G}, T)$ is topologically conjugate to a full-shift on a finite number of symbols.
\end{corollary}

\section{The Scale Function}\label{scale}

Several ideas that have been used arose independently in dynamics and in the study of the structure of l.c.t.d. groups.  One was stated in Lemma~\ref{subgroup}.  Below, some of these ideas are discussed.
In the study of l.c.t.d. groups the definitions of the scale function, tidy subgroups and their relationships are due to G.~A.~Willis and can be found in \cite{W1} and \cite{W2}.

Let $T$ be an automorphism of an l.c.t.d. group $\mathcal{G}$.  The number of cosets of $T(\mathcal{H}) \cap \mathcal{H}$ in $T(\mathcal{H})$ is the index of $T(\mathcal{H}) \cap \mathcal{H}$ in $T(\mathcal{H})$ and is denoted by $[T(\mathcal{H}): T(\mathcal{H}) \cap \mathcal{H}]$.  From the dynamics perspective this is the cardinality of the follower set of $\mathcal{H}$.  If $\mathcal{H}$ is compact and open, the index is finite.

\begin{observation}\label{index}
Let $\mathcal{H}$ be a compact, open subgroup in an l.c.t.d. group $\mathcal{G}$ and $T$ an automorphism of $\mathcal{G}$.
The following properties of the subgroup, $\mathcal{H}$, are equivalent.
\begin{enumerate}
\item The coset partition defined by $\mathcal{H}$ is memoryless.
\item The subgroup $\mathcal{H}$ is \textit{tidy above}.
\item $[T(\mathcal{H}_{loc}^u): \mathcal{H}_{loc}^u] = [T(\mathcal{H}): T(\mathcal{H}) \cap \mathcal{H}]$.
\item $\mathcal{H}  = \mathcal{H}^u_{loc} \mathcal{H}^s_{loc}  =\mathcal{H}^s_{loc} \mathcal{H}^u_{loc}$.
\item $\mathcal{H} = \mathcal{H}_{loc}^u (\mathcal{H} \cap T^{-1}(\mathcal{H}))$.
\item $T(x \mathcal{H}^s_{loc}) \subseteq T(x) \mathcal{H}^s_{loc}$ and $T(x) \mathcal{H}^u_{loc} \subseteq T(x \mathcal{H}^u_{loc})$, for all $x \in \mathcal{G}$.
\end{enumerate}
\end{observation}

The definition of tidy above and the equivalence of $(3)$, $(4)$ and $(5)$ are from the work of Willis.  The definition of memoryless partition or, in the expansive case, Markov partition is from dynamics where $(3)$
 is used as the definition and condition $(6)$ is the algebraic version of the geometric condition used when constructing Markov partitions for diffeomorphisms of manifolds.

\begin{definition} The \textit{scale} of $T$ is
\[
s(T) = \min \{ [T( \mathcal{H}): T(\mathcal{H}) \cap \mathcal{H}] \},
\]
where the minimum is taken over all compact, open subgroups $\mathcal{H}$ of $\mathcal{G}$.
A compact, open subgroup that attains the minimum is called \textit{minimizing}.
\end{definition}

The scale of an automorphism of a compact, totally disconnected group is one and the scale of an automorphism of a countable, discrete group is one.

\begin{definition} Let $T$ be an automorphism of the l.c.t.d. group $\mathcal{G}$ and $\mathcal{H}$ be a compact, open subgroup whose coset partition is memoryless.
Define the subgroups
\[
\mathcal{H}^s = \bigcup_{j \leq 0} T^j (\mathcal{H}_{loc}^s) \qquad \text{and} \qquad \mathcal{H}^u = \bigcup_{j \geq 0} T^j (\mathcal{H}_{loc}^u).
\]
In dynamical terms these are the \textit{stable set of the identity} and the \textit{unstable set of the identity}.
\end{definition}
Each is an increasing union of subgroups and so each is a subgroup of $\mathcal{G}$.

\begin{example}\label{3adic2} In Example \ref{3adic} $\mathcal{H}^u$ is the single point of all $\bar{0}$'s while $\mathcal{H}^s$ is the entire group. \end{example}

\begin{example} In Example \ref{3adicsq} $\mathcal{H}^u$ consists of all points that come from paths in the subgraph of the transition graph that is the tree rooted at $( \bar{0}, \bar{0})$ grown with forward transitions including the self-loop at $( \bar{0}, \bar{0})$ and $\mathcal{H}^s$ consists of all points that come from paths in the subgraph of the transition graph that is the tree rooted at $( \bar{0}, \bar{0})$ grown with backward transitions including the self-loop at $( \bar{0}, \bar{0})$.  Each is closed because the complement is easily seen to be open.
\end{example}

\begin{definition}\label{tbelow} A compact, open subgroup $\mathcal{H}$ of $\mathcal{G}$ is \textit{tidy below} if both $\mathcal{H}^s$ and $\mathcal{H}^u$ are closed in $\mathcal{G}$.
\end{definition}

\begin{definition}\label{tidy} A compact, open subgroup $\mathcal{H}$ of $\mathcal{G}$ is \textit{tidy} if it is tidy above and tidy below.
\end{definition}

\begin{theorem}\label{theoremW} A compact, open subgroup $\mathcal{H}$ of $\mathcal{G}$ is minimizing if and only if it is tidy.
\end{theorem}

\begin{example}\label{3adics} In Example \ref{3adic} $s(\sigma) = 1$ and $s(\sigma^{-1}) = 3$.  In the Example \ref{3adicsq} $s(\sigma) = s(\sigma^{-1}) = 3$. \end{example}

\begin{observation}\label{scalec}  Let $T$ be an automorphism of the l.c.t.d. group $\mathcal{G}$.  If the points with compact $T$-orbit closure are dense in $\mathcal{G}$ then $(\Sigma_{\mathcal{G}/\mathcal{H}}, \sigma)$ is topologically conjugate to the product of a full-shift on a finite alphabet and a permutation of a countable discrete set where every orbit is finite and the scale function has value 1.
\end{observation}

\begin{observation}\label{scalew}  Let $T$ be an automorphism of the l.c.t.d. group $\mathcal{G}$.  If the points with compact $T$-orbit closure are not dense in $\mathcal{G}$ then the scale function of $T$  and $T^{-1}$ are the cardinalities of the follower and predecessor sets, respectively, in any of the essentially wandering countable state Markov shifts of Theorem~\ref{theorem2}.
\end{observation}

\begin{proof}  Consider the proof of Theorem ~\ref{theorem1} and the essentially wandering countable state Markov shift produced there.  There is the special element of the alphabet coming from the $\mathcal{H}$ we have constructed to get the product structure.  The cardinalities of the of the follower sets and the predecessor sets in the Markov shift are the same for all elements of the alphabet of the shift.  In the Markov shift, $\mathcal{H}^u$ is a countable state Markov shift described by a tree rooted at the symbol from $\mathcal{H}$.  The symbol from $\mathcal{H}$ has a self loop and the rest of the graph is a tree.  It is a tree by Observation~\ref{tree}.  The points in the countable state Markov shift that arise from following such a tree form a closed set since there is no recurrence.  The inverse image of the full-shift cross the points from the tree is a closed set.  The same holds for $\mathcal{H}^s$.  The compact, normal subgroup which is the inverse image of the full-shift cross the special symbol is a tidy subgroup.
\end{proof}

Note that this implies that the cardinalities of the follower and predecessor sets are the same in every essentailly wandering Markov shift in the inverse limit of Theorem~\ref{theorem2}

\section{Transitivity, Ergodicity and Entropy}

Here we will examine some further consequences of Theorem~\ref{theorem2}.  We start with the following observation.

\begin{observation}\label{transitivefs} Let $T$ be an automorphism of the l.c.t.d. group $\mathcal{G}$.  If $T$ is transitive then $(\mathcal{G}, T)$ is topologically conjugate to $( \Sigma_\mathcal{G}, \sigma_\mathcal{G} )$ where $\Sigma_\mathcal{G}$ is an inverse limit of full shifts on finite alphabets and $\sigma_\mathcal{G}$ is the map induced by the shift transformation on every level.
\end{observation}

\begin{proof}  If $T$ is transitive then at every level of the inverse limit, $\mathcal{G} = \mathcal{G}_c$ and $F$ consists of a single fixed point. \end{proof}

\begin{corollary}\label{transitivec} \cite{A} Let $T$ be an automorphism of the l.c.t.d. group $\mathcal{G}$.  If $T$ is transitive then $\mathcal{G}$ is compact. \end{corollary}

When $T$ is transitive we have $(\Sigma_{N_n}, \sigma_n)$ at each level.  A factor map between two full-shifts is either boundedly finite-to-one or uncountably infinite-to-one on some point.  A boundedly finite-to-one factor map between full-shifts  preserves topological entropy and the full-shifts have the same number of symbols.  If the factor map is infinite-to-one on some point then the topological entropy is decreased by the logarithm of an integer and this integer is the difference in the cardinalities of the alphabets of the two full-shifts.

\begin{example}\label{fullshift} $\Sigma_2$ is a topological group when addition is defined symbol by symbol modulo 2 and $\sigma$ is an automorphism.  Consider the group homomorphism $\varphi : \Sigma_2 \rightarrow \Sigma_2$ defined by $\varphi (x)_i = x_i + x_{i+1} \pmod 2$ for each $i$.  It is a continuous factor map.  The factor map $\varphi$ is two-to-one on each point.  Forming the inverse limit by applying $\varphi$ at every level results in a totally disconnected, compact group, $\Sigma_\mathcal{G}$,  with an automorphism, $\sigma_\mathcal{G}$, that is transitive, nonexpansive and has topological entropy $\log 2$.
\end{example}

\begin{example}\label{infinent} $\Sigma_{2^n}$ is a topological group when addition is defined symbol by symbol modulo $2^n$ and $\sigma_n$ is an automorphism.  Consider the map $\varphi_n : \Sigma_{2^{n+1}} \rightarrow \Sigma_{2^n}$ defined by $\varphi_n (x)_i = x_i \pmod {2^n}$ for each $i$.  As before it is a continuous factor map which is also a group homomorphism.  The factor map $\varphi_n$ is uncountably infinite-to-one on each point.  Forming the inverse limit by applying $\varphi_n$ at each level results in a totally disconnected, compact group, $\Sigma_\mathcal{G}$,  with an automorphism, $\sigma_\mathcal{G}$, that is transitive, nonexpansive and has infinite topological entropy.
\end{example}

\begin{observation}\label{transitivefs} Let $T$ be an automorphism of the l.c.t.d. group $\mathcal{G}$.  If $T$ is ergodic with respect to Haar measure then $(\mathcal{G}, T)$ is topologically conjugate to
$( \Sigma_\mathcal{G}, \sigma_\mathcal{G} )$ where $\Sigma_\mathcal{G}$ is an inverse limit of full shifts on finite alphabets and $\sigma_\mathcal{G}$ is the map induced by the shift transformation on every level.
\end{observation}

\begin{proof}  The same reasoning applies as in the case where $T$ is transitive. \end{proof}

\begin{corollary}\label{transitivec} \cite{A} Let $T$ be an automorphism of the l.c.t.d. group $\mathcal{G}$.  If $T$ is ergodic with respect to Haar measure then $\mathcal{G}$ is compact. \end{corollary}

\end{document}